\documentclass[11pt]{article}
\usepackage[utf8]{inputenc}
\usepackage[T1]{fontenc}
\usepackage{comment}
\usepackage{authblk}
\usepackage{relsize}
\usepackage{changepage}
\usepackage{mathtools}
\usepackage[english]{babel}
\usepackage{amsmath}
\usepackage{amsthm}
\usepackage{amsfonts}
\usepackage{amssymb}
\usepackage{graphicx}
\usepackage[usenames,dvipsnames]{color}
\theoremstyle{plain}
\newtheorem{theorem}{Theorem}[section]
\newtheorem{corollary}[theorem]{Corollary}
\newtheorem{claim}[theorem]{Claim}

\newtheorem{lemma}[theorem]{Lemma}
\newtheorem{conjecture}[theorem]{Conjecture}

\usepackage{empheq}

\makeatletter
\newcommand{\vast}{\bBigg@{4}}
\newcommand{\Vast}{\bBigg@{5}}
\makeatother

\theoremstyle{definition}

\usepackage[left=2cm,right=2cm,top=2cm,bottom=2cm]{geometry}
\usepackage{amssymb}
\usepackage{hyperref}
\makeatletter
\def\namedlabel#1#2{\begingroup
    #2%
    \def\@currentlabel{#2}%
    \phantomsection\label{#1}\endgroup
}
\usepackage[normalem]{ulem}
\usepackage{dsfont}
\usepackage{accents}
\usepackage{verbatim}
\usepackage{ulem}
\usepackage{pgf,tikz,pgfplots}
\pgfplotsset{compat = 1.16}
\usepackage[all]{xy} 
\usepackage[tightpage]{preview}

\newcommand{\Aa}{\mathcal A}

\usepackage{subfig}

\usepackage{enumerate}

\title{\scshape
  Annulus graphs in $\mathbb R^d$}

\author[1,2]{Lyuben Lichev}
\author[3]{Tsvetomir Mihaylov}

\affil[1]{Universit\'e Jean Monnet, Saint-Etienne, France}
\affil[2]{Institute of Mathematics and Informatics, Bulgarian Academy of Sciences, Sofia, Bulgaria}
\affil[3]{University of Birmingham, Birmingham, UK}

\begin{document}

\maketitle
 
\begin{abstract}
A $d$-dimensional annulus graph with radii $R_1$ and $R_2$ (here $R_2\ge R_1\ge 0$) is a graph embeddable in $\mathbb R^d$ so that two vertices $u$ and $v$ form an edge if and only if their images in the embedding are at distance in the interval $[R_1, R_2]$. In this paper we show that the family $\mathcal A_d(R_1,R_2)$ of $d$-dimensional annulus graphs with radii $R_1$ and $R_2$ is \emph{uniquely} characterised by $R_2/R_1$ when this ratio is sufficiently large. Moreover, as a step towards a better understanding of the structure of $\mathcal A_d(R_1,R_2)$, we show that $\sup_{G\in \mathcal A_d(R_1,R_2)} \chi(G)/\omega(G)$ is given by $\exp(O(d))$ for all $R_1,R_2$ satisfying $R_2\ge R_1 > 0$ and also $\exp(\Omega(d))$ if moreover $R_2/R_1\ge 1.2$.
\end{abstract}

\section{Introduction}
In this paper, we are interested in graphs constructed as follows. Fix two real numbers $R_1,R_2 \ge 0$ satisfying $R_2\ge R_1$. The family of \emph{$d$-dimensional annulus graphs with radii $R_1$ and $R_2$}, denoted by $\Aa_d(R_1,R_2)$, consists of the graphs $G$ whose vertex set can be embedded in $\mathbb R^d$ so that, for all pairs of different vertices $u,v\in V(G)$, $uv$ is an edge of $G$ if and only if the distance between $u$ and $v$ in the embedding is in the interval $[R_1,R_2]$. We call any such embedding an $(R_1,R_2)$-\emph{annulus embedding}, or just an annulus embedding, of the graph $G$. In the sequel annulus graphs will be assumed finite unless explicitly defined as infinite.

The motivation for studying annulus graphs is twofold. To begin with, the notion interpolates between two classical models in graph theory: unit disc graphs and unit distance graphs. The family of $d$-dimensional unit disc graphs coincides with the family $\Aa_d(0,R)$ for every $R > 0$. Unit disc graphs were introduced in 1971 by Gilbert~\cite{Gil71} to model telecommunication networks. Since then, the most significant developments of the theory of unit disc graphs were made in the framework of random unit disc graph also known as random geometric graphs, see Penrose~\cite{Pen03} for a detailed account. The model was generalized by Waxman~\cite{Wax88}: he worked in a setting where two vertices in positions $x$ and $y$ are connected with probability $\beta \exp(-|x-y|/r)$ where $\beta$ and $r$ are parameters of the model. Penrose~\cite{Pen16} introduced and studied a percolated version of the model.

Another line of research was initiated in 1946 by Erd\H{o}s who was interested in the largest number of edges in a graph whose vertices may be embedded in $\mathbb R^2$ so that two vertices are at distance 1 in the embedding if and only if they form an edge in the graph. Erd\H{o}s~\cite{Erd46} himself was able to show a lower bound of the form $n^{1+c/\log\log n}$ for some constant $c > 0$ and offered a 500 dollar prize for a proof whether or not there is an upper bound of the same form. To our knowledge, the best currently known upper bound is proportional to $n^{4/3}$ and was provided by Spencer, Szemer\'edi and Trotter~\cite{SST84}. Another related problem is the Hadwiger-Nelson problem asking for the smallest number of colours in which the points of the plane may be coloured so that no two points at distance 1 are monochromatic. It has long been known that that the answer is between 4 and 7, and de Grey~\cite{deG18} improved the lower bound to 5 (result reproved independently by Exoo and Ismailescu~\cite{EI20}).

Another major motivation of the project is the (rather general) Goldilock's principle, which roughly states that ``objects that interact with each other should be neither too close nor too far from each other''. The principle appears in many areas such as cognitive science~\cite{Kid12} (infants prefer to occupy themselves with tasks that are neither too complex nor too simple), astronomy (the habitable zone around a star must be neither too close nor too far from it), economy (balancing high economic growth with low inflation) and machine learning~\cite{BL17} (concerning the learning rate of an algorithm) among others. 

Despite the interest attracted by the subject in many fields of science, to our knowledge rigorous mathematical analysis of annulus graphs was conducted only by Galhotra, Mazumdar, Pal and Saha~\cite{GMPS18}. There, the authors studied the threshold of connectivity of the $d$-dimensional random annulus graph $\mathcal G(a(\log n)^{1/d},b(\log n)^{1/d})$ with $a,b = O(1)$, which is obtained as follows. Consider a $d$-dimensional cube of side length $n^{1/d}$ and embed $n$ vertices uniformly at random. Then, connect two vertices if the distance between them is in the interval $[a(\log n)^{1/d}, b(\log n)^{1/d}]$. They show that there is a function $\varphi: \mathbb N\to \mathbb R_+$ such that, if $b^d-a^d < \varphi(d)$, then the graph is disconnected whp, while if $b^d-a^d > \varphi(d)$, then the graph is connected whp. A wider class of random intersection graphs with more general connection functions was studied in~\cite{DG16}.
\subsection{Our results}
While previous research concentrates on random annulus graphs, our results are deterministic in nature. Before presenting our first theorem, let us observe that, for any $d\in \mathbb N$ and $R_2 \ge R_1 > 0$, $\mathcal A_d(0,1)\neq \mathcal A_d(R_1,R_2)$. To see this, note that the family $\mathcal A_d(0,1)$ contains the family of all finite complete graphs while the clique number of every graph in $\mathcal A_d(R_1,R_2)$ is bounded.\footnote{Indeed, consider an $(R_1, R_2)$-\emph{annulus embedding} of a complete graph in $\mathbb R^d$: then, the vertices of the graph must be embedded at the centers of disjoint balls of radii $R_1/2$ which themselves are contained in a ball of radius $R_2$ (and center any of the embedded vertices, say). 
Thus, the smaller balls of radii $R_1/2$ pack a larger ball of radius $R_2+R_1/2$, and consequently there are at most $\left(\frac{R_2+R_1/2}{R_1/2}\right)^d$ smaller balls.} Another easy remark is that, for all $d\in \mathbb N$ and $R_2 \ge R_1 > 0$, $\mathcal A_d(R_1, R_2) = \mathcal A_d(1, R_2/R_1)$. The next result points in the direction of distinguishing the families $\mathcal A_d(1,R)$.

\begin{theorem}\label{thm 1}
For every $d\in \mathbb N$ there is a constant $C = C(d) > 0$ such that, for every pair of distinct real numbers $x,y \ge C$, $\mathcal A_d(1,x)\not\subseteq \mathcal A_d(1,y)$.
\end{theorem}

To push our investigation of the families $\mathcal A_d(1,R)$ further, we study the question whether or not these families are $\chi$-bounded. A family of graphs $\mathcal F$ is \emph{$\chi$-bounded} if there exists a function $f:\mathbb N\to \mathbb N$ such that, for every graph $G\in \mathcal F, \chi(G)\le f(\omega(G))$, where $\chi(G)$ denotes the chromatic number of $G$ and $\omega(G)$ denotes the size of the largest clique in $G$. Identifying $\chi$-bounded classes has become a hot trend in recent years, for a complete account we direct the reader to the outstanding survey of Scott and Seymour~\cite{SS20} on the subject. Particular attention was paid to \emph{intersection graphs}: given a collection $\mathcal F$ of sets, the intersection graph of $\mathcal F$ has vertex set $\mathcal F$ and edge set $\{(X,Y): X,Y\in \mathcal F, X\cap Y\neq \emptyset\}$. Asplund and Gr\"unbaum~\cite{AG60} showed that the family of intersection graphs of axis-parallel rectangles in the plane $\mathbb R^2$ is $\chi$-bounded. Surprisingly, this is not the case for 3-dimensional boxes as observed by Burling~\cite{Bur65} - he provided an explicit construction of a sequence of intersection graphs of boxes with a bounded clique number and chromatic number that tends to infinity. Intersection graphs of discs were shown to form a $\chi$-bounded class~\cite{BHMRR, Gra95} (in particular, for every unit disc graph $G$ one has $\chi(G)\le 6\omega(G)-5$). Moreover, for every unit disc graph $G$, Peeters~\cite{Pee91} showed that $\chi(G)\le 3\omega(G)-2$. His proof is based on an algorithm that, given an embedding of the vertex set of $G$ witnessing that $G$ is a unit disc graph, sweeps the points from bottom to top and greedily attributes a colour to a vertex in the moment when it is met. The proof of the upper bound in the next theorem uses a similar idea although the fact that an annulus is not a convex shape in general leads to some complications and requires additional ideas. For the lower bound, we construct a concrete embedding of a graph based upon a discretisation of the unit sphere $\mathbb S^{d-1}$.

\begin{theorem}\label{thm 2}
There exist constants $M,m > 1$ such that, for every $d\ge 2$:
\begin{enumerate}[(i)]
    \item\label{pt 1} for every $x\ge 1$ we have  
    \begin{equation*}
    \sup_{G\in \mathcal A_d(1,x)}\dfrac{\chi(G)}{\omega(G)}\le M^d;
    \end{equation*}
    \item\label{pt 2} for every $x\ge 1.2$ we have  
    \begin{equation*}
    m^d\le \sup_{G\in \mathcal A_d(1,x)}\dfrac{\chi(G)}{\omega(G)}.
    \end{equation*}
\end{enumerate}
Both bounds hold for unit disc graphs as well. Moreover, for $d = 1$ and every $x > 1$ one may ensure that $3/2 \le \sup_{G\in \mathcal A_1(1,x)} \chi(G)/\omega(G)\le M$.\footnote{The graphs in $\mathcal A_1(0,1)$ are also called unit interval graphs. It is well-known that the clique number and the chromatic number of interval graphs are equal, see e.g. Section 5.5 in~\cite{Die05}. Also, the connected unit distance graphs in $\mathbb R$ are paths so the same conclusion holds for them as well.}
\end{theorem}

\paragraph{Outline of the proofs.} In the proof of Theorem~\ref{thm 1} we consider a particular graph defined via a $(1,x)$-embedding in $\mathbb R^d$ and show that it does not admit a $(1,y)$-embedding. The proof is by contradiction and is divided into two cases: $d\ge 2$ and $d=1$. In the more substantial case $d\ge 2$, we first show that ``most pairs of vertices'' at distance at most 1 in the $(1,x)$-embedding must be at distance at most 1 in every $(1,y)$-embedding of the graph. Then, we show that ``most pairs of vertices'' at distance more than $x$ in the $(1,x)$-embedding must be at distance more than $y$ in every $(1,y)$-embedding of the graph. On the basis of these results we distinguish the cases $x < y$ and $x > y$ and reach a contradiction in each of them.

The proof of the upper bound of Theorem~\ref{thm 2} is based on the analysis of a geometric exploration algorithm. On the other hand, the proof of the lower bound is primarily based on an upper bound on maximal packings of the unit sphere with spherical caps.

\paragraph{Plan of the paper.} In Section~\ref{sec pf thm 1} we prove Theorem~\ref{thm 1}. In Section~\ref{sec pf thm 2} we prove Theorem~\ref{thm 2}. Section~\ref{sec discussion} is dedicated to a related discussion.

\section{Proof of Theorem~\ref{thm 1}}\label{sec pf thm 1}
For any real numbers $x\ge 1$ and $\varepsilon > 0$, denote by $G_{\infty,d} (x, \varepsilon)$ the graph that admits the following $(1,x)$-annulus embedding: its vertex set is embedded at the points $\{(\varepsilon n_i)_{i=1}^d: n_1, n_2, \dots, n_d \in \mathbb Z\}$ (a set which we denote $\varepsilon \mathbb Z^d$ in the sequel) and its edge set consists of the pairs $\{\{v_1,v_2\}: |v_1 - v_2|_d\in [1,x]\}$ (here, $|\cdot|_d$ denotes the Euclidean distance). We call this the \emph{natural} embedding of $G_{\infty,d}(x,\varepsilon)$. We will show that for all sufficiently large $x,y$ satisfying $x\neq y$ one may choose a small enough $\varepsilon > 0$ so that $G_{\infty,d}(x,\varepsilon)$ contains a finite subgraph outside $\mathcal A_d(1,y)$. Note that the case $d=1$ requires certain modifications, and is therefore treated in a simpler way in the end of the section. For now, we assume that $d\ge 2$. In the sequel, we simplify the notation $|\cdot|_d$ to $|\cdot|$, and denote by $\mathrm{vol}(\cdot)$ the volume function in $\mathbb R^d$.

For every $n\in \mathbb N$, denote by $G_{n,d}(x,\varepsilon)$ the graph defined by restricting the natural embedding of $G_{\infty,d} (x, \varepsilon)$ to the ball of center $(0,0)$ and radius $n$. Fix sufficiently large $x$ and $y$ and a sufficiently small $\varepsilon > 0$ and suppose for contradiction that for all $n\in \mathbb N$ one has $G_{n,d}(x,\varepsilon)\in \mathcal A_d(1,y)$. We start with a couple of preliminary results. The first of them says that, roughly speaking, most pairs of vertices of $G_{n,d} (x, \varepsilon)$ at distance at most 1 in the natural embedding are at distance at most 1 in any $(1,y)$-embedding of $G_{n,d} (x, \varepsilon)$.

\begin{lemma}\label{lem 1}
Fix a sufficiently large $C_0 = C_0(d) > 0$, $x, y\ge C_0$ satisfying $x\neq y$, an integer $n\ge x$ and any sufficiently small $\varepsilon > 0$. Then, for every pair of vertices $v_1,v_2$ of $G_{n,d} (x, \varepsilon)$ whose images $p_1,p_2$ in the natural embedding satisfy $|p_1-p_2| \le 1$ and $\max\{|p_1|,|p_2|\}\le n-x$, the images of $v_1$ and $v_2$ in every $(1,y)$-annulus embedding of $G_{n,d} (x, \varepsilon)$ in $\mathbb R^d$ are at distance at most 1.
\end{lemma}

We will need some preparation before presenting the proof of Lemma~\ref{lem 1}.

\begin{lemma}\label{lem 1.1}
Fix $C_1 > 0$, $x > 2 C_1+1$, any $\gamma\in ((C_1+1)/x,1)$ sufficiently close to $1$, an integer $n\ge x$ and any sufficiently small $\varepsilon > 0$. Also, fix a vertex $v$ of $G_{n,d}(x,\varepsilon)$ such that its image $p$ in the natural embedding satisfies $|p|\le n-x$, and fix a point $q$ on the boundary of $B(p, \gamma x)$. Denote by $S$ the set of points of $\varepsilon \mathbb Z^d$ in the intersection of $B(p,\gamma x)$ and $B(q,C_1)$. Then, for every $c_1 > 0$ there is $c_2 = c_2(d, \gamma, c_1) > 0$ such that every subset of $S$ of size $c_1 |S|$ contains a complete graph on $\lfloor c_2 C_1^d\rfloor$ vertices.
\end{lemma}

\begin{proof}
Denote $\mathcal R = B(p, \gamma x)\cap B(q,C_1)$. Note that $\mathcal R$ is disjoint from $B(p,1)$ as seen in Figure \ref{fig1}. Since $\gamma x > C_1 + 1$, $\mathcal R$ contains a cone $\mathcal C$ with center $q$, radius $C_1$ and aperture $2\pi/3$. Thus, in particular, for every sufficiently small $\varepsilon$ the number of points in $\mathcal R$ is at least $c C_1^d \varepsilon^{-d}$, where $c = c(d)$ is any constant smaller than the ratio of $\mathrm{vol}(\mathcal C)$ and $\mathrm{vol}(B(0,C_1))$. Note that $c$ is chosen independently of $C_1$.

Now, fix any $c_1\in (0,1]$ and let $\widehat S\subset S$ be a subset of $S$ of size at least $c_1 |S|$ points. Note that, since the diameter of $\mathcal R$ is bounded from above by $2C_1 \le x$, every point in $\widehat S$ must be connected to all points in $\widehat S$ but at most $(\varepsilon^{-1}+1)^d \mathrm{vol}(B(0,1))$, which is at most $(2\varepsilon^{-1})^d \mathrm{vol}(B(0,1))$ for all $\varepsilon\in (0,1/2]$. Thus, for all such $\varepsilon$ one may construct greedily a set of at least 
\begin{equation*}
    \dfrac{|\widehat S|}{(2\varepsilon^{-1})^d \mathrm{vol}(B(0,1))}\ge \dfrac{c_1 c C_1^d}{2^d \mathrm{vol}(B(0,1))}
\end{equation*}
vertices in $\widehat S$ which form a clique. This shows the lemma for $c_2 =  c_1 c/(2^d\, \mathrm{vol}(B(0,1))$.
\end{proof}

In the sequel, for every $\gamma \in (0,1)$ denote by $N_\gamma$ the maximal number of points in $B(0, \gamma)$ that are pairwise at distance more than 1.

\begin{figure}[ht] 
\centering
\includegraphics[width=\textwidth]{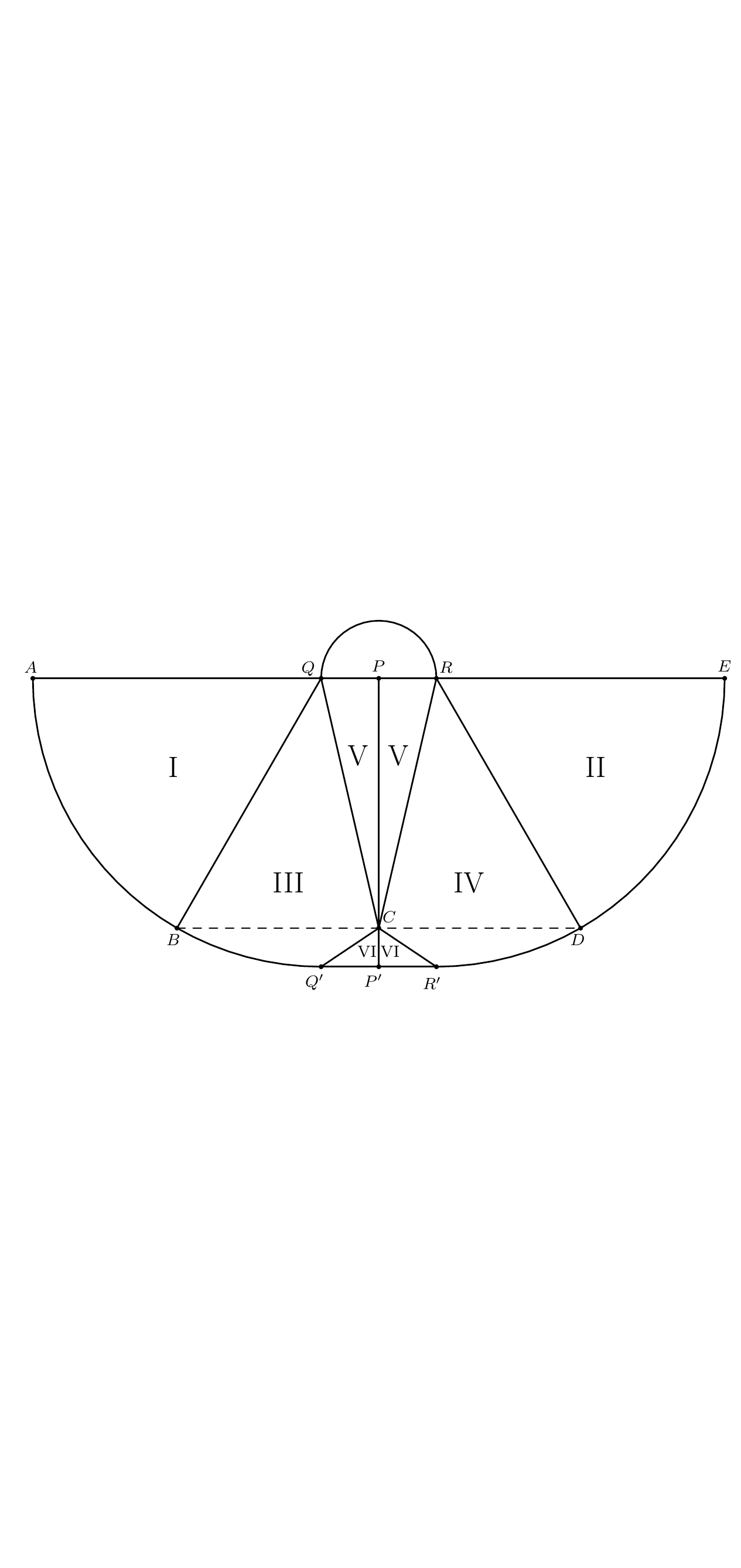}
\caption{The configuration from the proof of Lemma~\ref{lem 1.1}.}
\label{fig1}
\end{figure}

\begin{lemma}\label{lem 1.2}
There is a sufficiently large constant $C_2 > 0$ such that the following holds. Fix $x \ge C_2$, an integer $n\ge x$, $\gamma\in (0,1)$ sufficiently close to 1 and any sufficiently small $\varepsilon > 0$. Also, fix a vertex $v$ of $G_{n,d}(x,\varepsilon)$ such that its image $p_v$ in the natural embedding satisfies $|p_v|\le n-x$. Then, for every $(1,y)$-embedding of $G_{n,d}(x,\varepsilon)$ there exist $N_\gamma$ vertices in the neighbourhood of $v$ in the graph $G_{n,d}(x,\varepsilon)$ which are pairwise at distance more than $y$ in this embedding and are pairwise at distance more than $x$ in the natural embedding of $G_{n,d}(x,\varepsilon)$.
\end{lemma}
\begin{proof}
First, notice that one may cover the ball with radius 1 with a finite number of $M = M(d)$ balls of radius $1/2$. Moreover, if the images of two vertices of $G_{n,d}(x,\varepsilon)$ in any $(1,y)$-embedding of this graph are in a ball of radius $1/2$, they are not connected by an edge. Hence, the vertices in no ball of radius 1 in any $(1,y)$-embedding of $G_{n,d}(x,\varepsilon)$ contain a clique of size $M+1$.

Next, consider a set of points $(q_i)_{i\in [N_\gamma]}$ in $B(p_v, \gamma x)$ for which 
$$\min_{i,j\in [N_\gamma], i\neq j} |q_i - q_j|$$
is maximised. By definition of $N_\gamma$ we know that this maximum is more than $x$, and even that by choosing $\gamma\in (0,1)$ sufficiently close to 1 one may ensure that the distance between any two points is at least $(2-\gamma)x$ (indeed, this way one may ensure that $2-\gamma$ is arbitrarily close to 1). 

Now, choosing $C_2$ so that $(1-\gamma)C_2 > 2\sqrt{C_2}$ means that any points in two different balls among $(B(q_i,\sqrt{C_2}))_{i=1}^{N_\gamma}$ are at distance more than $x$. Also, for all $i\in [N_\gamma]$, denote by $S_i$ the set of vertices with image in $\varepsilon \mathbb Z^2\cap (B(p_v,\gamma x)\cap B(q_i, \sqrt{C_2}))$ in the natural embedding. Note that each $S_i$ is disjoint from $B(p_v,1)$ as seen in Figure \ref{fig2}. Now, choosing $c_1 = N_\gamma^{-1}$ and $c_2 = c_2(d, \gamma, c_1)$ as in Lemma~\ref{lem 1.1}, choosing $C_2 > c_2^{-1} (M+1)$ and using Lemma~\ref{lem 1.1} for any $i\in [N_\gamma]$ and any subset $\widehat S_i\subset S_i$ satisfying $|\widehat S_i|\ge |S_i|/N_\gamma$ gives that there are $M+1$ vertices in $\widehat S_i$ that induce a complete graph. Fix any vertex $u$ with image $p_u\in \bigcup_{i=1}^{N_\gamma} (B(p_v, \gamma x)\setminus B(p_v,1))\cap B(q_i, \sqrt{C_2})$ in the natural embedding of $G_{n,d}(x,\varepsilon)$. By the preceding observation, for any $i\in [N_\gamma]$ and in any $(1,y)$-embedding of $G_{n,d}(x,\varepsilon)$, $B(p_u,1)$ contains less than $|S_i|/N_\gamma$ vertices of $S_i$.

\begin{figure}[ht] 
\centering
\hbox{\hspace{4em} \includegraphics[width=\textwidth]{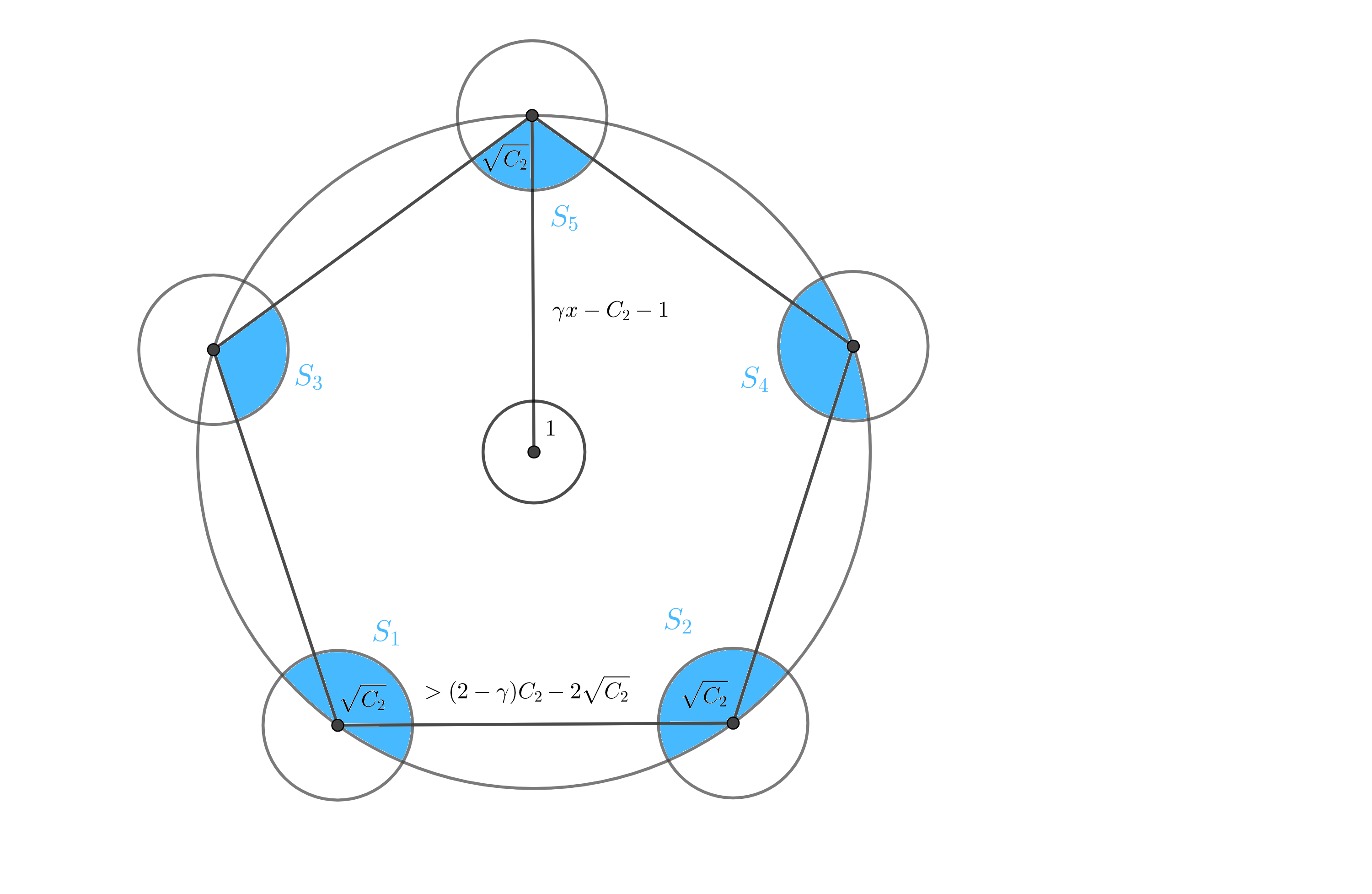}}
\caption{The configuration from the proof of Lemma~\ref{lem 1.2}.}
\label{fig2}
\end{figure}

Finally, fix any $(1,y)$-embedding of the neighbourhood of $v$. Now, for all $j\in [N_\gamma-1]$, choosing any vertices $v_1,\dots, v_j$ from $S_1,\dots, S_j$ ensures that there remain at least $(1-j/N_\gamma)|S_{j+1}|$ vertices in $S_{j+1}$, which are at distance at least $y$ from each of $v_1,\dots,v_j$. This allows us to choose greedily a set of $N_\gamma$ vertices, one from each of $(S_j)_{j=1}^{N_\gamma}$, whose images in the $(1,y)$-embedding are pairwise at distance more than $y$, which completes the proof.
\end{proof}

\begin{lemma}\label{3points}
Consider two points $a, b\in \mathbb R^d$ which are at distance more than $1$. Then, if $d\ge 3$ (respectively $d=2$), for every sufficiently large $\gamma\in (0,1)$ there is no set of $N_\gamma$ (respectively $3$) points in $B(a,1)\cap B(b,1)$ that are pairwise at distance larger than $1$.
\end{lemma}
\begin{proof}
Let $S$ be a subset of points in $B(a,1)\cap B(b,1)$ such that all pairs of points in $S$ are at distance more than 1. Suppose first that $d\ge 3$ (see Figure~\ref{fig63}). Then, there is $\gamma_1 \in (0,1)$ sufficiently close to 1 such that, up to slight modifications of the positions of the points in $S$ if necessary, one may assume that $S\subseteq B(a,\gamma_1)\cap B(b,\gamma_1)$. Set $\gamma = (1+\gamma_1)/2$. Then, on the one hand, $B(a, \gamma_1)\subseteq B(a,\gamma)$, and moreover the intersection point of the ray $ba^{\to}$ and the sphere $\partial B(a,\gamma)$, which is further from $b$, is at distance at least $\gamma+(1-\gamma_1) > 1$ from $S$. Hence, $B(a,\gamma)$ contains a point at distance more than $1$ from every point in $S$, and by definition of $N_\gamma$ we conclude that $N_\gamma > |S|$, which concludes the proof when $d\ge 3$.

\begin{figure}[ht]
\centering
\hbox{\hspace{4em} \includegraphics[scale=0.43]{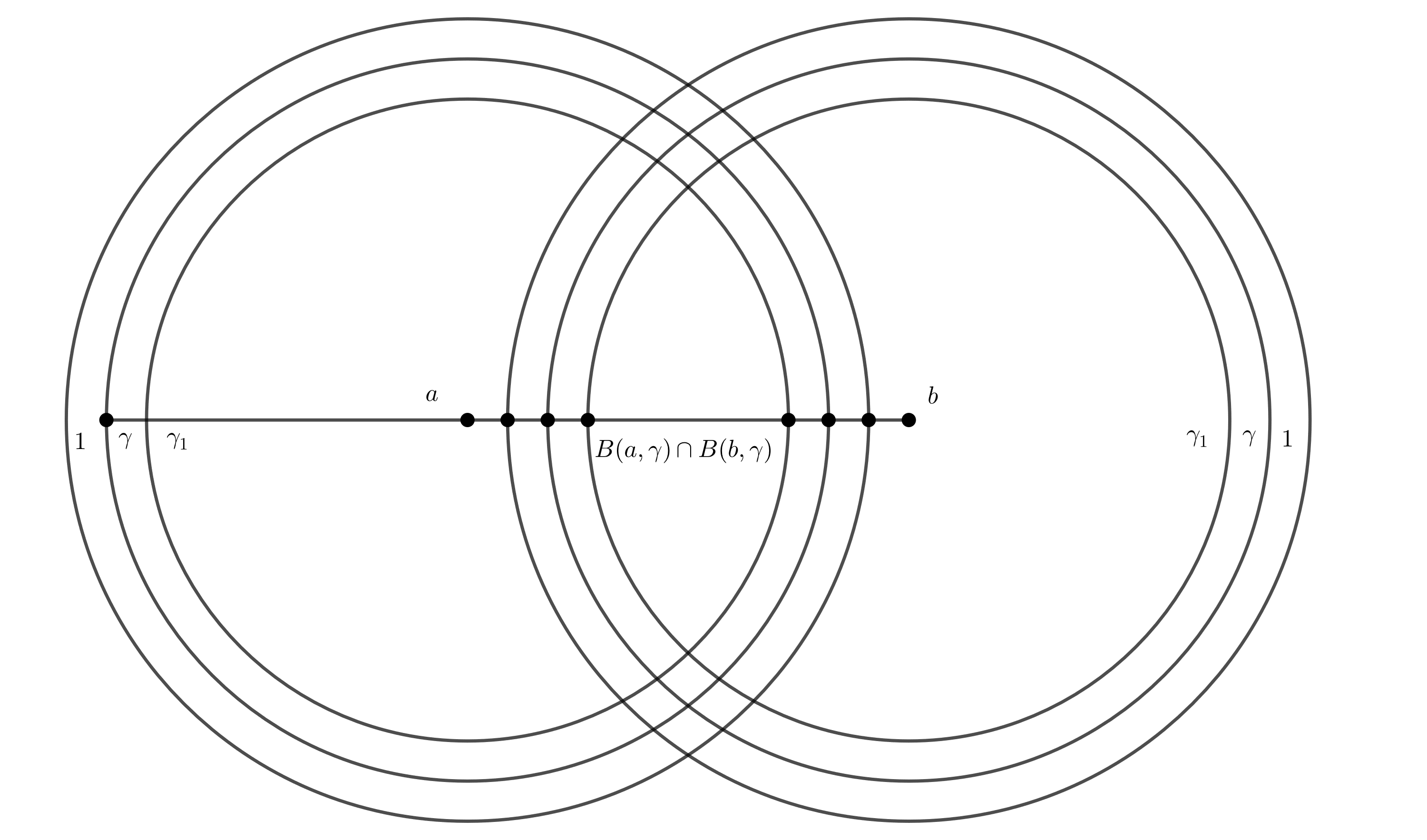}}
\caption{The configuration from the proof of Lemma~\ref{3points} for $d\ge 3$ (depicted in two dimensions for simplicity); $\gamma_1$, $\gamma$ and $1$ indicate the radii of the spheres in the figure.}
\label{fig63}
\end{figure}

Now, suppose that $d=2$ (see Figure~\ref{fig64}). Let the circles with centers $a$ and $b$ intersect at points $c$ and $c'$, and let the segment $ab$ intersect these circles at points $e$ and $f$. Then, for each point $p$ in the curved triangle $cef$ (where $ef$ is a segment and $ce$ and $cf$ are arcs of the unit circles), the ball $B(p, 1)$ covers the curved triangle $cef$. Hence, there can be at most one point in $S$ in the curved triangle $cef$. Similarly, there can be at most one point in $S$ in the curved triangle $c'ef$. Since $B(a,1)\cap B(b,1)$ is covered by these two triangles, there can be at most two points at a distance larger than $1$ there, which completes the proof. 

\begin{figure}[ht]
\centering
\hbox{\hspace{4em} \includegraphics[width=\textwidth]{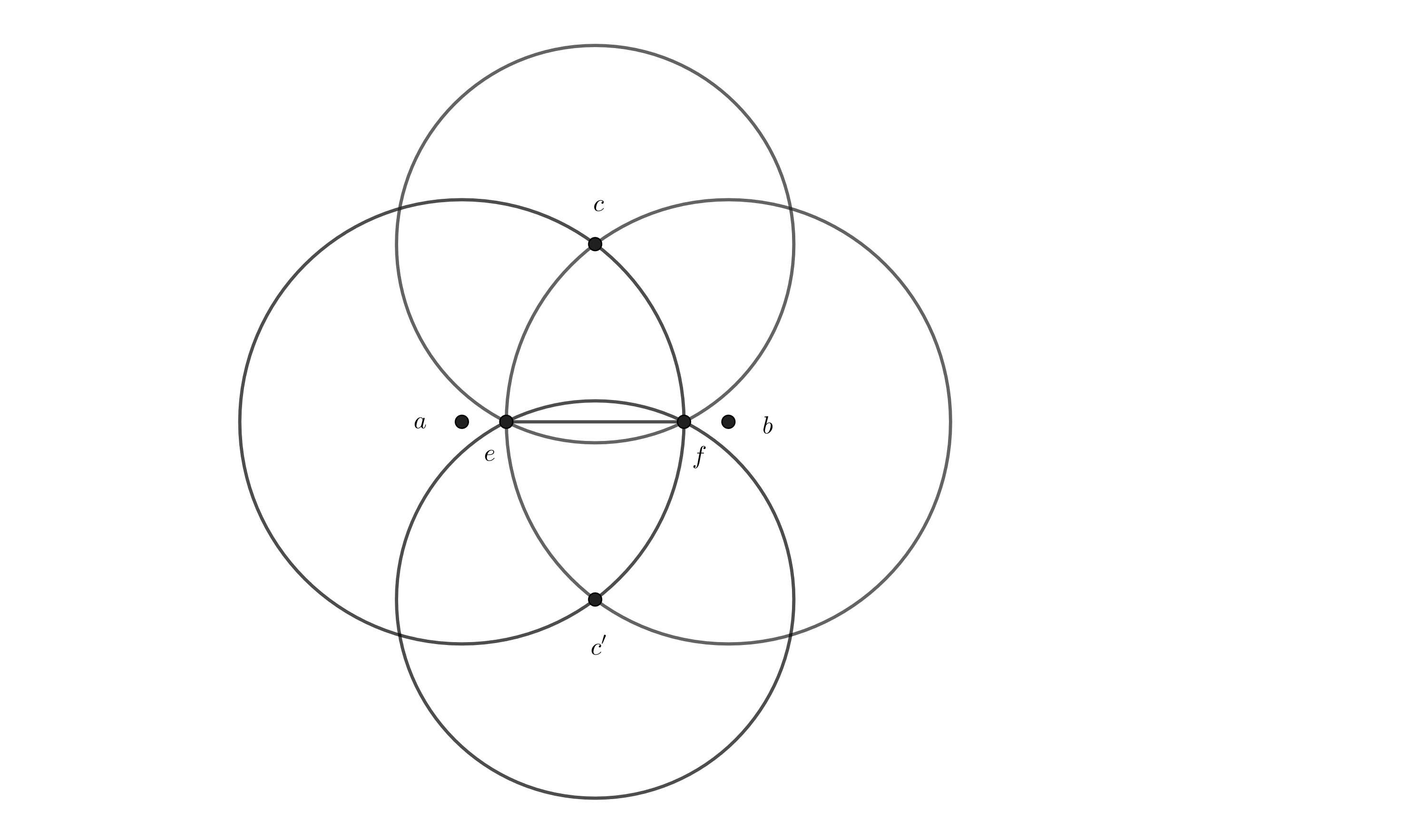}}
\caption{The configuration from the proof of Lemma~\ref{3points} for $d = 2$.}
\label{fig64}
\end{figure}
\end{proof}

We are ready to come back to the proof of Lemma~\ref{lem 1}.

\begin{proof}[Proof of Lemma~\ref{lem 1}]
Fix $\gamma\in (0,1)$ as in Lemma~\ref{3points} and $C_0 > (1-\gamma)^{-1}$. This ensures that $B(p_2, x)$ contains $B(p_1, \gamma x)$ (indeed, every point at distance at most $\gamma x$ to $p_1$ is at distance no more than $\gamma x + 1 < x$ from $p_2$). Thus, up to choosing $\varepsilon$ sufficiently small and $C$ sufficiently large, by Lemma~\ref{lem 1.2} in any $(1,y)$-embedding of $G_{n,d}(x,\varepsilon)$ there are $N_\gamma$ common neighbours to $v_1$ and $v_2$ that are embedded into points, which are pairwise at distance more than $y$. 

Now, assume for contradiction that in some $(1,y)$-embedding of $G_{n,d}(x,\varepsilon)$ the vertices $v_1$ and $v_2$ are embedded at distance more than $y$. Then, the intersection of the balls with radius $y$, centered at the images of $v_1$ and $v_2$ in the $(1,y)$-embedding (which we denote by $q_1$ and $q_2$), cannot contain $N_\gamma$ points at pairwise distances more than $y$ by Lemma \ref{3points}, which finishes the proof.
\end{proof}

We showed that, roughly speaking, ``most pairs of vertices that are close'' in the natural embedding of $G_n(x, \varepsilon)$ are still ``close'' in any $(1,y)$-embedding of $G_{n,d}(x, \varepsilon)$. A second step is to show that ``most'' pairs of vertices that are far in the natural embedding of $G_{n,d}(x, \varepsilon)$ are far in any $(1,y)$-embedding of $G_{n.d}(x, \varepsilon)$.

\begin{lemma}\label{cor of lem 1}
Fix a sufficiently large $C_3 > C_0$ (with $C_0$ defined in Lemma~\ref{lem 1}), $x, y\ge C_3$ satisfying $x\neq y$, an integer $n\ge x+1$ and any sufficiently small $\varepsilon > 0$. Then, for every pair of vertices $v_1,v_2$ of $G_{n,d} (x, \varepsilon)$ whose images $p_1,p_2$ in the natural embedding satisfy $|p_1-p_2| > x$ and $\max\{|p_1|,|p_2|\}\le n-x-1$, the images of $v_1$ and $v_2$ in any $(1,y)$-annulus embedding of $G_{n,d} (x, \varepsilon)$ in $\mathbb R^d$ are at distance more than $y$.
\end{lemma}

Before proceeding with the proof, we need to prepare the ground with two technical results.

\begin{lemma}\label{easy lemma}
For every $\gamma\in [0.99,1)$, we have $N_\gamma\ge 5$ if $d=2$ and $N_\gamma\ge 2d+2$ if $d\ge 3$.
\end{lemma}
\begin{proof}
If $d = 2$, inscribing a regular pentagon in the sphere $\partial B(0, 0.99)$ is sufficient since the side of this pentagon is $0.99\cdot 2 \sin(\pi/5) > 1$. Suppose that $d\ge 3$. Then, consider the set $S$ of the points 
\begin{align*}
& a_1 = (x,y,0,\ldots,0),\; a_2 = (x,-y,0,\ldots,0),\;a_3 = (-x,y,0,\ldots,0),\; a_4 = (-x,-y,0,\ldots,0),\\
& b_1 = (x,0,y,\ldots,0),\; b_2 = (x,0,-y,\ldots,0),\;b_3 = (-x,0,y,\ldots,0),\; b_4 = (-x,0,-y,\ldots,0),\\
& \forall i\in [4, d], z_{i,1} = -z_{i,2} = (0, \ldots, 0, 1, 0, \ldots, 0),
\end{align*}

\begin{figure}[ht] 
\centering
\hbox{\hspace{2em} \includegraphics[width=45em]{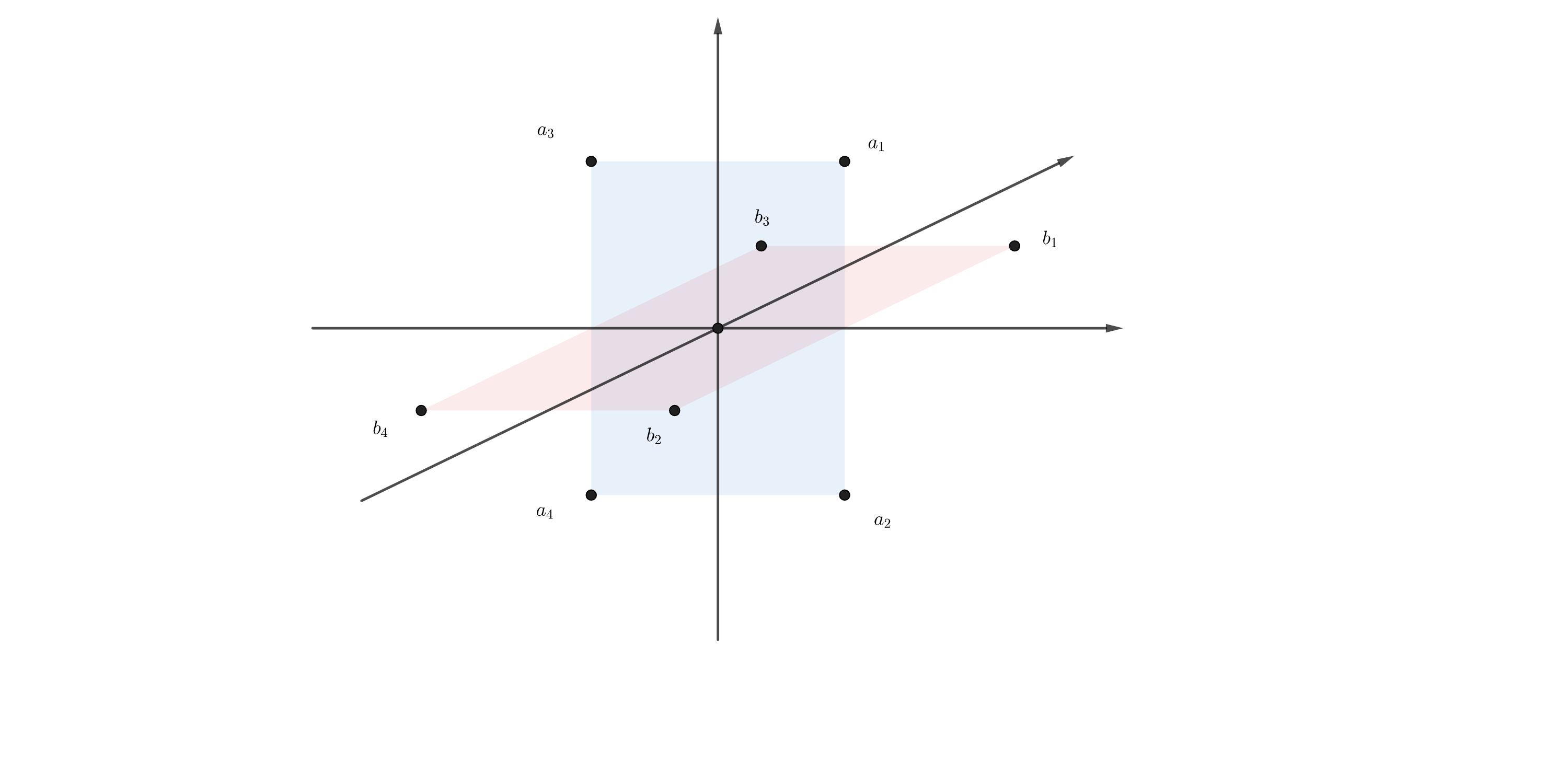}}
\caption{The configuration from the proof of Lemma~\ref{easy lemma}.}
\label{fig5}
\end{figure}
\noindent 
where the unique 1 in $z_{i,1}$ is in the $i$-th position.
We choose positive $x$ and $y$ so that $x^2+y^2 = 0.99^2$ and $x = 0.6$. Then, $y = \sqrt{0.99^2 - 0.6^2} \approx 0.79 > x$ and so
$$\min_{i,j\in [4], i\neq j} |a_ia_j| = 2|x| = 1.2 \text{ and } \min_{i,j\in [4]} |a_ib_j| = \sqrt{2y^2} \approx 1.24 > 1.$$
Moreover, since for all $i\in [4,d]$ and $j\in \{1,2\}$, the minimum to the distances from $z_{i,j}$ to any other point in $S$ is more than 1, we conclude that $N_{\gamma}\ge |S| = 2d+2$, which proves the lemma.
\end{proof}

The second technical result we need appears as Theorem~2 in~\cite{M91} (which is formulated in terms of a graph parameter called \emph{sphericity} of the complete bipartite graph $K_{d,d}$); see also Section~5 in~\cite{RRS89} for the related notion of spherical dimension.

\begin{lemma}[see Theorem~2 in~\cite{M91}]\label{lem cyclicity}
For every $d\ge 1$, there is no set of\, $2d+2$ points $\{p_1, \ldots, p_{2d+2}\}$ in $\mathbb R^d$ such that:
\begin{itemize}
    \item $|p_i-p_j| > 1$ for all $i,j$ such that either both $i,j\ge d+2$ or both $i,j\le d+1$;
    \item $|p_i-p_j| \le 1$ for all $i,j$ such that $1 \le i \le d+1$ and $d+2 < j \le 2d+2$.\qed
\end{itemize}
\end{lemma}

\noindent
We are ready to prove Lemma~\ref{cor of lem 1}.

\begin{proof}[Proof of Lemma~\ref{cor of lem 1}]
Choosing $C_3 > \max\{C_0,3\}$, where $C_0$ was chosen in Lemma~\ref{lem 1}, we have that $B(p_1,1)\cap B(p_2,1) = \emptyset$ and moreover one may find unit half-balls $B^1_{1/2}\subset B(p_1,1)$ and $B^2_{1/2}\subset B(p_2,1)$ satisfying $B^1_{1/2}\cap B(p_2,x) = B^2_{1/2}\cap B(p_1,x) = \emptyset$. Also, choose $\gamma\in [0.99,1)$ as in Lemma~\ref{3points}, and fix a set $S$ of $N_\gamma$ points at distance more than 1 in $B(0,1)\setminus \{0\}$. Now, for $\varepsilon$ sufficiently small there is an injective map $\phi: B(0,1)\to B^1_{1/2}\cup B^2_{1/2}$ for which $\phi(S)$ consists of $|S|$ points at distance more than 1, from which $k := \lfloor N_\gamma/2\rfloor$ are in $B^1_{1/2}$ and $N_{\gamma} - k$ are in $B^2_{1/2}$. Such a map indeed exists: for example, consider the hyperplane orthogonal to the vector $(1,0,\ldots,0)\in \mathbb R^d$, and then rotate it in the plane, generated by the first two coordinates. By discrete continuity of the difference of the number of points in $S$ on the two sides there is a moment when:
\begin{itemize}
    \item there in no point lying on the rotating hyperplane, and
    \item the difference between the numbers of points in $S$ on the two sides is at most 1.
\end{itemize} 
At this point, it is sufficient to ``split'' $B(0,1)$ into two halves and map them to $B^1_{1/2}$ and to $B^2_{1/2}$ in the natural way (the choice of an image for the points on the hyperplane itself is arbitrary).

Now, we argue by contradiction. Suppose that in some $(1,y)$-embedding of $G_{n,d}(x,\varepsilon)$ the images of $v_1$ and $v_2$, which we denote by $q_{v_1}$ and $q_{v_2}$, respectively, are at distance at most 1. Denote by $(q_i)_{i = 1}^{k}$ (respectively $(q_i)_{i = k+1}^{N_\gamma}$) the images of the vertices, corresponding to the points in $\phi(S)\cap  B^1_{1/2}$ (respectively in $\phi(S)\cap  B^2_{1/2}$) in the given $(1,y)$-embedding. Then, since $B(q_{v_1}, 1)\cup B(q_{v_2}, 1)$ has diameter at most $3 \le y$, we have that
\begin{itemize}
    \item $|q_i-q_j| > 1$ for all $i,j$ such that either both $i,j\ge k+1$ or both $i,j\le k$;
    \item $|q_i-q_j| \le 1$ for all $i,j$ such that $1 \le i \le k$ and $k+1 \le j \le N_{\gamma}$.
\end{itemize}

If $d\ge 3$, then $k\ge d+1$ by Lemma~\ref{easy lemma} and therefore such set of points does not exist in $\mathbb R^d$ by Lemma~\ref{lem cyclicity}, which is a contradiction with our assumption. If $d = 2$, then $k\ge 2$ and $N_{\gamma}-k \ge 3$, which again leads to contradiction with the statement for $d=2$ in Lemma~\ref{3points}.
\end{proof}

\begin{corollary}\label{cor1}
Fix $C_3 > 0$ as in Lemma~\ref{cor of lem 1}, $x, y > C_3$ satisfying $x\neq y$ and integers $n\ge x$ and $k > 1$. Then, for every pair of vertices $v_1,v_2$ of $G_{n,d} (x, \varepsilon)$ whose images $p_1,p_2$ in the natural embedding satisfy $|p_1-p_2|< k$ and $\max\{|p_1|,|p_2|\}\le n-x-1$, the images of $v_1$ and $v_2$ in every $(1,y)$-annulus embedding of $G_{n,d} (x, \varepsilon)$ in $\mathbb R^d$ are at distance at most $k$ for every sufficiently small $\varepsilon > 0$ (depending only on $p_1, p_2$ and $k$).
\end{corollary}
\begin{proof}
Fix $\delta = k - |p_1 - p_2|$ and $\varepsilon \in (0, (5k)^{-1}\delta]$. Divide the segment between $p_1$ and $p_2$ into $k$ segments $x_0x_1 = p_1x_1, x_1x_2,\dots, x_{k-2}x_{k-1}, x_{k-1}x_k = x_{k-1}p_2$ of equal length. Then, for all $i\in [k]\cup \{0\}$, associate to $x_i$ a nearest vertex $x'_i\in \varepsilon \mathbb Z^d$ (so $x'_0 = x_0$ and $x'_k = x_k$). Clearly for all $i\in [k]\cup \{0\}$ one has $|x_i - x'_i|\le 2\varepsilon$ and $|x'_i|\le 2\varepsilon + x_i\le 2\varepsilon + n - x - 1 < n - x$. By the triangle inequality $|x'_i - x'_{i-1}|\le 1 - k^{-1}\delta + 2\cdot 2\varepsilon < 1$, which means that $(x'_{i-1}x'_i)_{i\in [k]}$ are all of length at most 1, so we conclude by applying Lemma~\ref{lem 1} $k$ times.
\end{proof}

\begin{corollary} \label{exxpand}
Fix $x, y > C_3$ satisfying $x\neq y$ and integers $n\ge x$ and $k>1$. Then, for every pair of vertices $v_1,v_2$ of $G_{n,d} (x, \varepsilon)$ whose images $p_1,p_2$ in the natural embedding satisfy $|p_1-p_2|< kx$, the images of $v_1$ and $v_2$ in every $(1,y)$-annulus embedding of $G_{n,d} (x, \varepsilon)$ in $\mathbb R^d$ are at distance at most $ky$ for every sufficiently small $\varepsilon > 0$ (depending only on $p_1, p_2, x$ and $k$).
\end{corollary}
\begin{proof}
Consider two vertices $v_1,v_2$ of $G_{n,d}(x,\varepsilon)$ whose images $p_1,p_2$ in the natural embedding are at a distance at most $x$. If they are at a distance at most $1$, then the images of $v_1,v_2$ in an $(1,y)$-annulus embedding of $G_{n,d}(x,\varepsilon)$ are also at a distance at most $1$ by Lemma \ref{lem 1} (and the fact that $C_3 > C_0$, as stated in Lemma~\ref{cor of lem 1}). If $p_1$ and $p_2$ are at a distance between $1$ and $x$, then we have that $v_1$ and $v_2$ are connected by an edge in $G_{n,d}(x,\varepsilon)$. Therefore, in any $(1,y)$-annulus embedding they should be at a distance between $1$ and $y$. We showed that if two vertices of $G_{n,d}(x,\varepsilon)$ have images in the natural embedding at a distance less then $x$, they should be at a distance less than $y$ in any $(1,y)$-annulus embedding. To finish the proof we do the same argument as in Corollary~\ref{cor1} but with steps of length $x$ and $y$ instead of steps of length $1$.
\end{proof}

For every $d\in \mathbb N$ and every pair of point sets $K,L\subseteq \mathbb R^d$, define $M_d(K,L)$ as the maximum number of disjoint congruent copies of $L$ included in $K$, that is, the maximum size of an $L$-packing of $K$.

\begin{proof}[Proof of Theorem \ref{thm 1} for $d\ge 2$]
We set $C = C_3$, where $C_3$ was provided in Lemma~\ref{easy lemma}, $x, y\geq C$ satisfying $x\neq y$,  any sufficiently large $n$ and any sufficiently small $\varepsilon>0$ (these two last parameters will be specified later). We will prove that $G_{n,d}(x,\varepsilon)$ cannot be realised in any $(1,y)$-annulus embedding. Denote by $\mathcal{X}$ the natural embedding of $G_{n,d}(x,\varepsilon)$ and fix some $(1,y)$-annulus embedding $\mathcal{Y}$ of $G_{n,d}(x,\varepsilon)$. We split the proof in two cases.

\noindent
{\bf{Case 1: $x<y$}}\\
Let $k$ be a positive integer. Consider a set of vertices $V$ in $G_{n,d}(x,\varepsilon)$ with image $V_x$ in $\mathcal X$, which is included in $B(0,k-1)$ and for every $p_1,p_2\in V_x$ we have $|p_1 - p_2| > x$. For $n > k+x+1$, by Lemma~\ref{cor of lem 1} the image $V_y$ of $V$ in $\mathcal{Y}$ is a set of points at a distance more than $y$ from each other. Moreover, by Corollary~\ref{cor1} and by choosing $\varepsilon$ sufficiently small, these points are contained in a closed ball of radius $k$. 

Fix a sufficiently small $\delta > 0$ and consider a packing $\mathcal B$ of the ball $B(0,k-1+x/2)$ with balls of radius $x/2+\delta$. By choosing a sufficiently small $\varepsilon = \varepsilon(\delta) > 0$, for each ball $B\in \mathcal B$ there is a point $u\in \mathcal{X}$ at a distance at most $\delta/2$ from the centre of $B$. Thus, from $\mathcal B$ one may construct a packing $\mathcal B_{\mathcal X}$ of $B(0,k-1+x/2)$ with balls of radius $x/2+\delta/2$ and centres in $\mathcal{X}$. Consequently, one may choose a set $V$ of size at least $M_d(B(0,k-1+x/2), B(0, x/2+\delta))$.

However, in $\mathcal Y$ the balls with radii $y/2$ and centers $V_y$ pack the ball $B(0,k+y/2)$. Moreover, since $x < y$, by choosing $\delta\le (y-x)/4$ we obtain that
\begin{equation}\label{m2ineq}
    M_d(B(0,k-1+x/2), B(0, (x+y)/4))\le M_d(B(0,k+y/2), B(y/2))
\end{equation}
for all $k\in \mathbb N$. However, for all fixed $r > 0$ the limit of $M_d(B(0,k), B(0,r)) \mathrm{vol}(B(0,r))/\mathrm{vol}(B(0,k))$ when $k\to +\infty$ exists and is given by the packing density of $\mathbb R^d$ with unit balls. Since
\begin{equation*}
    \lim_{k\to +\infty}\frac{\mathrm{vol}(B(0,k-1+x/2))}{\mathrm{vol}(B(0,k+y/2))} = 1,
\end{equation*}
we conclude that $\mathrm{vol}(B(0, (x+y)/4))\ge \mathrm{vol}(B(0, y/2))$, which leads to a contradiction.

\noindent
{\bf{Case 2: $x>y$}}\\
Let $k$ be a positive integer. Consider a set of vertices $V$ in $G_{n,d}(x,\varepsilon)$ with image $V_x$ in $\mathcal X$, which is included in $B(0,kx-1)$ and for every $p_1,p_2\in V_x$ we have $|p_1-p_2| > 1$. For $n>k+x+1$ Lemma~\ref{lem 1} implies that the image $V_y$ of $V$ in the embedding $\mathcal{Y}$ is a set of points at a distance more than $1$ from each other, and by Corollary~\ref{exxpand} and by choosing $\varepsilon$ sufficiently small these points are contained in a ball of radius $ky$.

In the exact same way as in the first case we find out that for every $\delta > 0$ one has
\begin{equation}\label{eq 1}
M_d(B(0,kx-1/2+\delta), B(0, 1/2+\delta))\le M_d(B(0,ky+1/2), B(0,1/2))
\end{equation}
Hence,
\begin{equation}\label{eq 2}
\begin{split}
&\lim_{k\to +\infty} \frac{M_d(B(0,kx-1/2+\delta), B(0, 1/2+\delta)) \mathrm{vol}(B(0, 1/2+\delta))}{\mathrm{vol}(B(0,kx-1/2+\delta))}\\
&\;\;\,\text{and}\quad \lim_{k\to +\infty} \frac{M_d(B(0,ky+1/2), B(0, 1/2)) \mathrm{vol}(B(0, 1/2))}{\mathrm{vol}(B(0,ky+1/2))}
\end{split}
\end{equation}
exist and are both equal to the packing density of $\mathbb R^d$ with unit balls. However,
\begin{equation}\label{eq 3}
    \lim_{k\to +\infty} \frac{\mathrm{vol}(B(0,kx-1/2+\delta))}{\mathrm{vol}(B(0,ky+1/2))} = \left(\frac{x}{y}\right)^d,
\end{equation}
so \eqref{eq 1}, \eqref{eq 2} and \eqref{eq 3} together imply that $x^{-d} \mathrm{vol}(B(0, 1/2+\delta))\ge y^{-d} \mathrm{vol}(B(0, 1/2))$, which leads to a contradiction by choosing $\delta$ sufficiently small. Thus, the proof of Theorem~\ref{thm 1} is completed.
\end{proof}

It remains to deal with the case $d=1$. Although the main points of the proofs are the same, the proof is technically simpler in this case.

\begin{proof}[Proof of Theorem~\ref{thm 1} for $d=1$]
Again, consider the graph $G_{n,d}(x, \varepsilon)$.

\begin{claim}\label{cl 1}
Lemma~\ref{lem 1} holds for $d=1$ as well.
\end{claim}
\begin{proof}[Proof of Claim~\ref{cl 1}]
We argue by contradiction. Let without loss of generality $p_1 < p_2$. Suppose that in some $(1,y)$-embedding, the images $q_1, q_2$ of $v_1, v_2$, respectively, are at distance more than $y$. Then, all common neighbours of $v_1$ and $v_2$ in $G_{n,1}(x,\varepsilon)$ have images in the $(1,y)$-embedding, contained in the intersection of the annuli around $q_1$ and $q_2$ with radii 1 and $y$, which is a segment of length less than $y$.

Now, divide the common neighbourhood of $v_1$ and $v_2$ into 4 groups: ones with images in the interval $[p_1-x, p_1 - x/2]$ in the natural $(1,x)$-embedding (set $V_1$), ones with images in the interval $(p_1-x/2, p_1]$ (set $V_2$), ones with images in the interval $[p_2, p_2+x/2)$ (set $V_3$) and ones with images in the interval $[p_2+x/2, p_2+x]$ (set $V_4$). Note that there are no edges between $V_2$ and $V_4$, between $V_4$ and $V_1$, and between $V_1$ and $V_3$. Hence, in the given $(1,y)$-embedding, the image of every vertex in $V_2$ (respectively in $V_4$ and in $V_1$) is at distance at most 1 from the image of every vertex in $V_4$ (respectively in $V_1$ and in $V_3$). Thus, for every sufficiently small $\varepsilon$, the images of $V := V_1\cup V_2\cup V_3\cup V_4$ are contained in an interval of length at most 4: indeed, fixing a vertex $v\in V_2$, the image of every vertex in $V$ in the $(1,y)$-embedding is at distance at most 2 from the image of $v$. However, this means that the set $V$ induces no $K_6$ in $G_{n,1}(x,\varepsilon)$, which is not the case for $x$ sufficiently large and $\varepsilon$ sufficiently small. 

\begin{figure}[ht] 
\centering
\hbox{ \includegraphics[width=\textwidth]{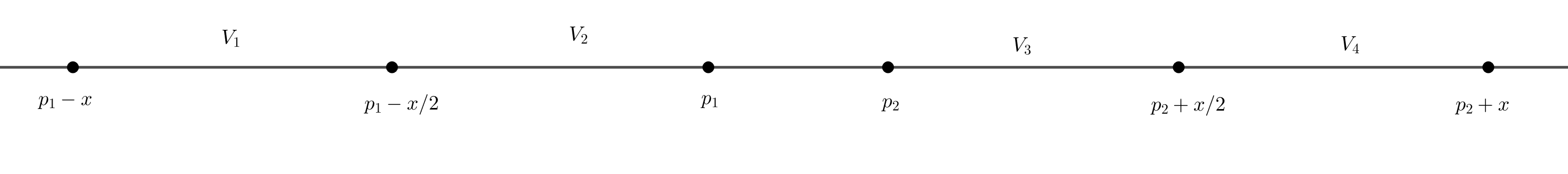}}
\caption{The configuration from the proof of Claim~\ref{cl 1}.}
\label{fig6}
\end{figure}
\end{proof}

\begin{claim}\label{cl 2}
Lemma~\ref{cor of lem 1} holds for $d=1$ as well.
\end{claim}
\begin{proof}[Proof of Claim~\ref{cl 2}]
We argue by contradiction. Let without loss of generality $p_1 < p_2$. Suppose that in some $(1,y)$-embedding, the images $q_1, q_2$ of $v_1, v_2$, respectively, are at distance less than 1. Then, if $y\ge 2$ and $\varepsilon$ is sufficiently small, one may conclude by Claim~\ref{cl 1} and an easy induction that for every vertex with position in the interval $(-\infty, p_1]$ in the $(1,x)$-embedding, its position in the $(1,y)$-embedding is at distance at most 1 from $q_2$. At the same time, for all sufficiently large $y$ and sufficiently small $\varepsilon$, this contradicts with the fact that the vertices with images in $(-\infty, p_1]$ in the $(1,x)$-embedding induce a graph, containing a $K_4$.
\end{proof}

The remainder of the proof is analogous to the proof in the case $d\geq 2$.
\end{proof}

\section{Proof of Theorem~\ref{thm 2}}\label{sec pf thm 2}

\subsection{Proof of the upper bound of Theorem~\ref{thm 2}} \label{chi_main}
In the heart of the proof of the upper bound of Theorem~\ref{thm 2} is the following algorithm, which colours the vertices of an annulus graph $G \in \mathcal A_d(r_1,r_2)$ properly (i.e. no two adjacent vertices share the same colour). First, given an annulus embedding of $G$, rotate it so that no two (images of) vertices in $V(G)$ have coinciding last coordinates. Then, fix an affine hyperplane orthogonal to the last coordinate axis of $\mathbb R^d$ and which is below the entire vertex set of $G$. Then, start moving this hyperplane continuously upwards. We colour the vertices of $G$ in colours indexed by the positive integers. When (the image of a) vertex in $G$ meets the moving hyperplane:
\begin{itemize}
    \item if this vertex has already been coloured, do nothing;
    \item if this vertex has not been coloured before, consider the set of uncoloured vertices at distance at most $r_1/2$ from it. Colour all of these vertices in the smallest colour which is still available for all of them.
\end{itemize}

At each step when the hyperplane meets a still uncoloured vertex $v$, the algorithm colours a set $S_v$ of previously uncoloured vertices (of course, $v\in S_v$). For every vertex $u\in S_v$, call $v$ the \emph{token} of $u$ and denote $v=t(u)$. Observe that, first, of all vertices coloured at the same moment the token has the smallest last coordinate, and second, vertices that have the same token also have the same colour. By construction the above algorithm produces a proper colouring of $G$, which we call $c_G$. Moreover, note that the particular case when $r_1 = 0$ coincides with Peeters' sweeping algorithm~\cite{Pee91}.

We proceed with a lemma that we will use in the proof of our theorem. In the sequel, we tacitly identify any $(r_1,r_2)$-annulus graph with an arbitrary $(r_1,r_2)$-annulus embedding of this graph in $\mathbb R^d$.

\begin{lemma}\label{lemcol}
Consider a graph $G\in\mathcal{A}_d(r_1,r_2)$. For every vertex $v$ in $G$, the ball $B(v,r_1)$ contains vertices in at most $7^d$ colours in $c_G$.
\end{lemma}
\begin{proof}
If $r_1 = 0$, the statement is trivial. Assume that $r_1 > 0$. Note that each vertex $u\in B(v,r_1)\cap V(G)$ satisfies $|u - t(u)|\le r_1/2$. Since $u$ is in $B(v,r_1)$, $t(u)$ is in $B(v,3r_1/2)$. Moreover, a token $t_2$ coloured later than a token $t_1$ in the algorithm must be at a distance more than $r_1/2$ from $t_1$ since otherwise the vertex $t_2$ would itself have $t_1$ as a token. Hence, for each pair of tokens $t_1$ and $t_2$, the balls $B(t_1,r_1/4)$ and $B(t_2,r_1/4)$ do not intersect. We conclude that the number of tokens that may fit into $B(v,3r_1/2)$ is at most $M_d(B(0,7r_1/4), B(0,r_1/4))\le 7^d$ and therefore the ball $B(v, r_1)$ contains points in at most $7^d$ different colours.
\end{proof}

We also make use of the following result. 

\begin{theorem}[simplified version of Theorems 1.1 and 1.2 in \cite{vgcover}, see also \cite{rcover}]\label{cover} 
Fix any $d\in \mathbb N$, $r, R > 0$ such that $R\ge r$, and let $T = R/r$. Let $\nu_{T,d}$ be the minimal number of closed balls of radius $r$ which may cover a closed ball of radius $R$ in $\mathbb{R}^d$. Then,
\begin{equation*}
    1\le \nu_{T,d}\le P(d) T^d,
\end{equation*}
where $P(d)$ is a polynomial function, which may be chosen of degree $3$.\qed
\end{theorem}

\begin{proof}[Proof of upper bound of Theorem \ref{thm 2}]
Consider $r_1\ge 0$ and $r_2 > 0$ satisfying $r_2\ge r_1$. Consider an $(r_1,r_2)$-annulus graph $G$ together with the (proper) colouring $c_G$ given by the colouring algorithm described above. Let $k = k(G)$ be the largest colour in $c_G$ and $t_k$ be a token with colour $k$. Note that the vertices which may forbid the colours in $[k-1]$ for $t_k$ are at distance at most $r_2$ from $B(t_k, r_1/2)$. Hence, all such vertices are contained in the lower half of $B(t_k, r_1/2 + r_2)$. 

Since $r_1 \leq r_2$, by Theorem \ref{cover} the ball $B(t_k,r_1/2+r_2)$ may be covered by at most $\nu_{T,d}\le (3+o(1))^d$ balls of radius $r_2/2$ where $T=2+r_1/r_2\le 3$. By the pigeonhole principle one of these balls of radius $r_2/2$, say $B$, contains vertices in at least $k/\nu_{T,d}$ colours of $c_G$. By Lemma~\ref{lemcol} one may find a set $S\subseteq B\cap V(G)$ of at least $k/(\nu_{T,d}\, 7^d)$ vertices in $B$ which are pairwise at distance more than $r_1$. Thus, the distance between each pair of vertices in $S$ is between $r_1$ and $r_2$, and hence $S$ is a clique. Therefore,
\begin{equation} \label{chi-bound}
\frac{k}{\nu_{T,d}\, 7^d}\leq \omega(G).
\end{equation}
Hence, $\chi(G)\le k\le (21+o_d(1))^d \omega(G)$, where $o_d(1)$ signifies a quantity that goes to $0$ as $d$ goes to $\infty$. Thus we have
\begin{equation*}
M := \sup_{d\in \mathbb N}\; \sup_{r_1, r_2} \sup_{G\in \mathcal A_d(r_1, r_2)} \left(\dfrac{\chi(G)}{\omega(G)}\right)^{1/d} < +\infty,
\end{equation*}
which proves the theorem. Note that here we proved the result also for unit disc graphs.
\end{proof}

\subsection{Proof of the lower bound}\label{sec LB}

Before providing the proof of the lower bound of Theorem~\ref{thm 2} in full generality, we show the case $d = 1$ as a warm-up. We recall that when $r_1 > 0$, we set $x=r_2/r_1$, and in particular $\mathcal A_d(r_1,r_2)=\mathcal A_d(1,x)$.

\begin{proof}[Proof of the lower bound of Theorem~\ref{thm 2} for $d = 1$]
Fix $x > 1$. We will show that $\mathcal A_1(1,x)$ contains a triangle-free graph with an odd cycle, which has chromatic number at least 3 and clique number 2. We consider two cases:
\begin{itemize}
    \item if $x\ge 2$, then the following five points form an embedding of the 5--cycle with vertices, listed in consecutive order: $0,x,2x,x+0.99,x-0.99$.
    \item if $x < 2$, then let $k\ge 2$ be the smallest integer such that $kx\ge k+1$. The following $2k+1$ points form an embedding of a graph containing a $(2k+1)$-cycle with vertices, listed in consecutive order: 
    \begin{equation*}
    0,x,\dots,kx,(k-k(k+1)^{-1}) x,(k-2k(k+1)^{-1}) x,\dots,(k-k^2(k+1)^{-1}) x.
    \end{equation*}
    Clearly all pairs of consecutive points as well as the first and the last point form edges. At the same time, every graph in $\mathcal A_1(1,x)$ is triangle-free since the sum of two numbers in the interval $[1,x]$ is at least $2 > x$.
\end{itemize}
This completes the proof in the case $d=1$.
\end{proof}

Fix any $d\ge 2$ and $x\in [1.2,+\infty)$. We first provide a proof of the lower bound for the family $\mathcal A_d(1,x)$ and then come back to the case of unit disc graphs (that is, $\mathcal A_d(0,1)$). The proof of the lower bound relies on a construction that approximates the infinite uncountable $(2/x,2)$-annulus graph with vertex set $\mathbb S^{d-1}$. Our main goal is to provide an example of an annulus graph in $\mathcal A_d(1,x)$ with a large (multiplicative) gap between its chromatic number and its clique number based on discrete versions of the following two theorems.

For any $d\ge 2$ and a set $X\subseteq \mathbb S^{d-1}$, the \emph{spherical diameter of $X$}, denoted $\phi_X$, is the supremum over all pairs of points $x,y\in X$ of the spherical distance between $x$ and $y$. The next theorem due to Schmidt~\cite{Sch48, Sch49} relates the diameters of two compact sets $X_1, X_2$ with the maximum distance between a point in $X_1$ and a point in $X_2$.

\begin{theorem}[isodiametric inequality, see Chapter~II, Section~8 in~\cite{Sch48}]\label{thm Schmidt}
Let $E$ be a metric space among $\mathbb S^{d-1}, \mathbb R^{d-1}$ and the $(d-1)$-dimensional hyperbolic space. Fix two compact sets $X_1, X_2\subseteq E$ and define $D$ as the supremum of the distance between $x_1$ and $x_2$ over all pairs of points $x_1\in X_1, x_2\in X_2$. Also, let $d_1$ (respectively $d_2$) be the diameter of a ball in the same space having volume equal to the volume of $X_1$ (respectively of $X_2$). Then, $d_1+d_2\le 2D$.
\qed
\end{theorem}

We will use the following corollary of Theorem~\ref{thm Schmidt} that was proved independently by B\"or\"oczky and Sagmeister, see Theorem~1.2 in~\cite{BS19}.

\begin{corollary}\label{thm sph cap}
Let $X$ be a measurable subset of $\mathbb S^{d-1}$ with spherical diameter $\phi_X < \pi$. Then, the Lebesgue measure of (the closure of) $X$ is at most the volume of a spherical cap with spherical diameter $\phi_X$.\qed
\end{corollary}

In the sequel, we denote by $\mathrm{Cap}^{d-1}(\phi)$ the spherical cap in $\mathbb S^{d-1}$ with center $(1,0,\ldots,0)$ and diameter $\phi$, and define $M(d,\phi)$ to be the maximum number of disjoint copies of $\mathrm{Cap}^{d-1}(\phi)$ that can be packed in $\mathbb S^{d-1}$, that is, $M(d,\phi) = M_d(\mathbb S^{d-1}, \mathrm{Cap}^{d-1}(\phi))$. The next theorem from~\cite{KL78} provides an upper bound on $M(d,\phi)$.

\begin{theorem}[\cite{KL78}, see also Section 2.4 in~\cite{TK93}]\label{thm proportion packing}
For all $\phi\in [0,\pi)$,
\begin{equation*}
    \dfrac{1}{d} \ln M(d,\phi)\le \dfrac{1+\sin \phi}{2\sin \phi} \ln\left(\dfrac{1+\sin \phi}{2\sin \phi}\right) - \dfrac{1-\sin \phi}{2\sin \phi} \ln\left(\dfrac{1-\sin \phi}{2\sin \phi}\right) +o_{d}(1).
\end{equation*}\qed
\end{theorem}

Now, we outline the idea of the proof. Fix $x\in [1.2,+\infty)$. We will provide a graph $G\in \mathcal A_d(2/x,2)$ for which $\chi(G)/\omega(G)\geq m^d$ for some constant $m>1$ independent of $d$ and $x$. For a graph $G$, we denote by $|G|$ the number of its vertices and by $\alpha(G)$ the size of its maximum independent set, that is, the size of the largest set of vertices inducing no edge. Recall that, for every non-empty graph $G$, $\chi(G)\geq |G|/\alpha(G)$. Therefore, we will provide a graph $G$ satisfying $|G|/(\alpha(G)\omega(G))\ge m^d$ and conclude by the previous observation. 

We start by giving an upper bound on the clique number of any annulus graph on the sphere $\mathbb S^{d-1}$. Then, we find a set of points in $\mathbb S^{d-1}$ forming an annulus graph $G_d$ with $|G_d|/\alpha(G_d)$ suitably bounded from below. 

\begin{lemma}
Consider a graph $G$ with an $(2/x, 2)$-embedding in the sphere $\mathbb S^{d-1}$. Then, its clique number is at most $M(d,2\arcsin(x^{-1}))$. 
\end{lemma}
\begin{proof}
Since the diameter of $\mathbb S^{d-1}$ is $2$, a pair of vertices of $G$ are connected by an edge if and only if they are at a distance at least $2/x$. Thus, a clique in $G$ is a set of vertices which are pairwise at Euclidean distance at least $2/x$, and in particular the open balls with radii $1/x$ around these points are disjoint. Hence, the largest clique in such a graph has size at most $M(d,2\arcsin(x^{-1}))$.
\end{proof}

For every integer $d\ge 2$, we define $\mathrm{svol}_{d-1}(X)$ as the $(d-1)$-dimensional Lebesgue measure on $\mathbb S^{d-1}$. The next lemma provides an upper bound for $M(d,2\arcsin(x^{-1}))$.

\begin{lemma}\label{lem analysis}
For all $x\in [1.2, +\infty)$, we have $M(d,2\arcsin(x^{-1}))\leq\dfrac{(0.997+o_d(1))^d\,\mathrm{svol}_{d-1}(\mathbb S^{d-1})}{\mathrm{svol}_{d-1}(\mathrm{Cap}^{d-1}(\arcsin(x^{-1})))}$.
\end{lemma}

\begin{proof}
It is well-known (see e.g. page 67 in~\cite{Li11}) that for all $\theta\in (0,\pi/2]$,
\begin{equation}\label{eq cap volume}
    \dfrac{\mathrm{svol}_{d-1}(\mathrm{Cap}^{d-1}(\theta))}{\mathrm{svol}_{d-1}(\mathbb S^{d-1})} = \dfrac{\int_0^{\theta} (\sin t)^{d-2} dt}{\int_0^{\pi} (\sin t)^{d-2} dt} = (\sin \theta + o_d(1))^d.
\end{equation}

\noindent
Note that for the second equality in~\eqref{eq cap volume}, we used Laplace's method to approximate $\int_0^b(\sin t)^{d-2}dt$ by $(\sin x_0+o_d(1))^d$, where $x_0$ is the unique maximum of $\sin$ in $(0,b]\subseteq (0,\pi]$. Set $\theta = \arcsin(x^{-1})$. By Theorem~\ref{thm proportion packing}, we have
\begin{equation*}
(\sin\theta)^d M(d,2\theta) \le \left(\sin\theta\exp\left(\dfrac{1+\sin(2\theta)}{2\sin(2\theta)} \ln\left(\dfrac{1+\sin(2\theta)}{2\sin(2\theta)}\right) - \dfrac{1-\sin(2\theta)}{2\sin(2\theta)} \ln\left(\dfrac{1-\sin(2\theta)}{2\sin(2\theta)}\right)\right) +o_d(1)\right)^d.
\end{equation*}

\noindent
A study of the function
\begin{equation*}
\theta\in [0,\arcsin(1.2^{-1})]\mapsto \sin\theta\exp\left(\dfrac{1+\sin(2\theta)}{2\sin(2\theta)} \ln\left(\dfrac{1+\sin(2\theta)}{2\sin(2\theta)}\right) - \dfrac{1-\sin(2\theta)}{2\sin(2\theta)} \ln\left(\dfrac{1-\sin(2\theta)}{2\sin(2\theta)}\right)\right)
\end{equation*}
shows that its maximum is attained at $\theta = \arcsin(1.2^{-1})$ and this maximum is smaller than $0.997$.\footnote{The reader may verify our claim by following the link \url{https://www.wolframalpha.com/input?i=maximize+sin\%28x\%29*exp\%28\%28\%281\%2Bsin\%282x\%29\%29\%2F\%282*sin\%282x\%29\%29\%29*ln\%28\%281\%2Bsin\%282x\%29\%29\%2F\%282*sin\%282x\%29\%29\%29-\%28\%281-sin\%282x\%29\%29\%2F\%282*sin\%282x\%29\%29\%29*ln\%28\%281-sin\%282x\%29\%29\%2F\%282*sin\%282x\%29\%29\%29\%29+for+x+in+\%5B0\%2C+arcsin\%281.2\%5E\%28-1\%29\%29\%5D+}.} This shows that $\dfrac{\mathrm{svol}_{d-1}(\mathrm{Cap}^{d-1}(x^{-1}))\,M(d,2\arcsin(x^{-1}))}{\mathrm{svol}_{d-1}(\mathbb S^{d-1})}\leq (0.997+o_d(1))^d$, from which the lemma follows.
\end{proof}

A set $\mathcal N_{\varepsilon}\subseteq \mathbb S^{d-1}$ is called an $\varepsilon$-net if every point in $\mathbb S^{d-1}$ has a point in $\mathcal N_{\varepsilon}$ at spherical distance at most $\varepsilon$. 

\begin{lemma}\label{lem net}
For any sufficiently small real number $\delta = \delta(x, d) > 0$, there is a $(2/x,2)$-annulus graph $G$ satisfying
\begin{equation*}
    \dfrac{|G|}{\alpha(G)}\geq\dfrac{\mathrm{svol}_{d-1}(\mathbb S^{d-1})}{\mathrm{svol}_{d-1}(\mathrm{Cap}^{d-1}(\arcsin(x^{-1}) + \delta))}.
\end{equation*}
\end{lemma}
\begin{proof}
We show that for every sufficiently small $\varepsilon = \varepsilon(\delta, d) > 0$, there is an $\varepsilon$-net $\mathcal N_{\varepsilon}$ in $\mathbb S^{d-1}$ which is the vertex set of an annulus graph with the required property. Fix a small enough $\varepsilon$ and
consider a tessellation $\mathcal T$ of $\mathbb S^{d-1}$ into regions of spherical diameter at most $\varepsilon$ and area at least $a = a(\delta, d, \varepsilon) > 0$. Consider a Poisson Point Process $\mathcal P$ with intensity $\lambda = \lambda(\delta, d, \varepsilon, a) > 0$ on the sphere satisfying that $\lambda a \ge (\log |\mathcal T|)^2$, where $|\mathcal T|$ stands for the number of regions in $\mathcal T$. By using well-known estimates for the tails of Poisson random variables (see e.g. Theorem~A.1.15 in~\cite{AS16}) gives that
\begin{align*}
& \mathbb P(\exists R\in \mathcal T: |R\cap \mathcal P| - \mathbb E |R\cap \mathcal P|\notin (-\varepsilon \mathbb E |R\cap \mathcal P|, \varepsilon \mathbb E |R\cap \mathcal P|))\\
\le\hspace{0.3em} 
& |\mathcal T| \max_{R\in \mathcal T} \mathbb P(|R\cap \mathcal P| - \mathbb E |R\cap \mathcal P|\notin (-\varepsilon \mathbb E |R\cap \mathcal P|, \varepsilon \mathbb E |R\cap \mathcal P|))\\
\le\hspace{0.3em} 
& |\mathcal T|\exp(-\Omega_{\varepsilon}(\log |\mathcal T|)^2) = o_{|\mathcal T|}(1).
\end{align*}
Thus, every region $R\in \mathcal T$ contains whp a number of points that is in the interval $((1-\varepsilon) \mathbb E |R\cap \mathcal P|, (1+\varepsilon) \mathbb E |R\cap \mathcal P|)$ and, in particular, at least one point. We condition on this event and set $\mathcal N_{\varepsilon} = \mathcal P$. Let $G$ be the $(2/x,2)$-annulus graph on vertex set $\mathcal N_{\varepsilon}$ and $I$ be any maximum independent set of $G$. Also, let $\mathcal T_I$ be the union of all regions of $\mathcal T$ containing a vertex in $I$. Then, since every region in the tessellation $\mathcal T$ has spherical diameter at most $\varepsilon$, 
the spherical diameter of $\mathcal T_I$ is at most $\varepsilon + 2\arcsin(x^{-1}) + \varepsilon = 2(\arcsin(x^{-1}) + \varepsilon)$. Hence, by choosing $\varepsilon < \delta/2$ one may derive by Corollary~\ref{thm sph cap} that $\mathcal T_I$ has $(d-1)$-dimensional volume at most $\mathrm{svol}_{d-1}(\mathrm{Cap}^{d-1}(\arcsin(x^{-1})+\delta/2))$. We conclude that every independent set of $G$ contains at most 
\begin{equation*}
    |\mathcal N_{\varepsilon}|\dfrac{(1+\varepsilon)\,\mathrm{svol}_{d-1}(\mathrm{Cap}^{d-1}(\arcsin(x^{-1})+\delta/2))}{(1-\varepsilon)\,\mathrm{svol}_{d-1}(\mathbb S^{d-1})}
\end{equation*}
points, which, up to choosing $\varepsilon$ sufficiently small, is bounded from above by
\begin{equation*}
    |\mathcal N_{\varepsilon}|\dfrac{\mathrm{svol}_{d-1}(\mathrm{Cap}^{d-1}(\arcsin(x^{-1})+\delta))}{\mathrm{svol}_{d-1}(\mathbb S^{d-1})},
\end{equation*}
which proves the lemma.
\end{proof}

\begin{proof}[Proof of Theorem~\ref{thm 2}~\eqref{pt 2}]
Fix a sufficiently small $\delta > 0$ and the corresponding graph $G$ given in Lemma~\ref{lem net}. Then,
\begin{equation} \label{1.2eq}
\dfrac{\chi(G)}{\omega(G)}\ge \dfrac{|G|}{\alpha(G)\,\omega(G)}\geq\dfrac{\mathrm{svol}_{d-1}(\mathbb S^{d-1})}{\mathrm{svol}_{d-1}(\mathrm{Cap}^{d-1}(\arcsin(x^{-1})+\delta))\, M(d,2\arcsin(x^{-1}))},
\end{equation}
where the first inequality holds for any graph and the second inequality follows from Lemma \ref{lem net}. Choosing $\delta>0$ small enough and using Lemma \ref{lem analysis}, we can ensure that the right hand side in~\eqref{1.2eq} is bounded below by $1.003^d$, which finishes the proof.
\end{proof}

It remains to deduce the result in the case of unit disc graphs. Fix $x\in [1.2,\infty]$. Consider a unit disc graph $G$ which is constructed as the annulus graph with radii $0$ and $2/x$ on a set of points lying on the sphere $\mathbb S^{d-1}$. Then, its complement $G^c$ is a graph for which two vertices are connected if the distance between them is in the interval $(2/x, 2]$. Note that every largest clique in a graph $G$ is a largest independent set in $G^c$ (that is, the complement of $G$) and every largest independent set of $G$ is a largest clique in $G^c$. Therefore,
\begin{equation*}
    \frac{\chi(G)}{\omega(G)}\geq \frac{|G|}{\alpha(G)\,\omega(G)}= \frac{|G^c|}{\alpha(G^c)\,\omega(G^c)}.
\end{equation*}
Moreover, the above proof of Theorem~\ref{thm 2} works with minor modifications for $(R_1, R_2)$-annulus graphs for which two points are connected if they are at distance $(R_1,R_2]$, which concludes the proof.

\section{Discussion}\label{sec discussion}
We finish with a short discussion around the bounds on $\sup_{G\in \mathcal A_d(1,x)}\chi(G)/\omega(G)$ provided by Theorem~\ref{thm 2}. To begin with, showing that a similar lower bound holds for $x\in [1,1.2)$ is an obvious open question. One construction similar to ours is to connect a point $p$ (see it as ``the north pole'') with all points at distance between $x_1$ and $x_2$ where $x_2/x_1 = x$ and $x_1^2 + x_2^2 = 4$ (this is a description of the points in $\mathbb S^{d-1}$ in a strip symmetric with respect to ``the equator''). The missing piece to show that this more universal construction provides an exponential lower bound for every $x > 1$ is the following conjecture of ours, which bears close resemblance to Kalai's double-cap conjecture, see Conjecture~2.8 in~\cite{FM86}.

\begin{conjecture}
Fix some $\theta\in [0,\pi/2)$. Every measurable set $X\subseteq \mathbb S^{d-1}$ containing no two points at spherical distance in the interval $[\pi/2 - \theta, \pi/2+\theta]$ has measure at most $2|\mathrm{Cap}^{d-1}(\pi/4 - \theta/2)|$, where the unique maximiser (up to rotation) is a pair of diametrically opposite caps with spherical diameters $\pi/2-\theta$.
\end{conjecture}

We remark that an approach similar to ours (but providing a worse lower bound for $x$) may be applied when $\theta$ is ``close'' to $\pi/2$ but not when $\theta$ is ``small''. It is worth observing that when $x = 1$ (that is, in the case of the unit distance graphs), we have that the largest cliques have size $d + 1$ while the chromatic number of a unit-distance graph may be larger than $1.2^d$ for all sufficiently large $d$ by a result of Frankl and Wilson~\cite{FW81}. 
Nevertheless, a randomised construction embedding points on the $d$-dimensional sphere with radius slightly larger than $1/\sqrt{2}$ uniformly at random implies that $(1,x)$-annulus graphs may contain cliques with exponential (in $d$) number of vertices for all $x > 1$.

Theorem~\ref{thm 2} shows that $\sup_{G\in \mathcal A_d(1,x)} \chi(G)/\omega(G)$ grows exponentially fast with $d$, which is somehow satisfactory for high dimensional annulus graphs. However, although it is easy to improve Lemma~\ref{lemcol} by a factor of 2, our approach seems to provide an upper bound that is far from optimal for small values of $d$. In the particular case of $\mathcal A_2(0,1)$ a simplified version of our algorithm coincides with the one used by Peeters~\cite{Pee91}, so in particular every graph $G\in \mathcal A_2(0,1)$ is shown to satisfy $\chi(G)\le 3\omega(G)-2$. However, if $\omega(G)=2$, we claim that $\chi(G)\le 3$: indeed, for any embedding of $G$ in $\mathbb R^2$ witnessing that $G\in \mathcal A_2(0,1)$, no two edges on four different vertices may intersect since otherwise by triangle inequality $G$ must contain a triangle with three out of these four vertices. Thus, all triangle-free graphs in $\mathcal A_2(0,1)$ are planar, so also 3-colourable by a theorem of Gr\"otzsch~\cite{Gro58} -- a bound attained by any cycle of odd length.

On the other hand, Malesi{\'n}ska, Piskorz and Wei{\ss}enfels~\cite{Mal96} showed that for every $\omega\in \mathbb N$ there is a graph $G\in \mathcal A_2(0,1)$ satisfying $\omega(G) = \omega$ and $\chi(G)\ge \lfloor 3\omega/2\rfloor$. Closing the gap between the lower and the upper bounds is a long-standing open problem.

\paragraph{Acknowledgements.} We are especially grateful to several anonymous referees for numerous useful comments and suggestions.

\bibliographystyle{plain}
\bibliography{References}

\end{document}